\documentclass[11pt]{amsart}

\usepackage{amsmath,bbm,a4, cite}

\newtheorem{theorem}{Theorem}

\newtheorem{lemma}[theorem]{Lemma}
\newtheorem{corol}{Corollary}[theorem]

\newcommand{\R}{\mathbb{R}}

\newcommand{\cK}{{\mathcal K}}
\newcommand{\cP}{{\mathcal P}}
\newcommand{\cT}{{\mathcal T}}

\newcommand{\cpni}{\cP^n_{\scriptscriptstyle 0}} 
\newcommand{\cpnii}{\cP^n_{\scriptscriptstyle (0)}} 

\newcommand{\ckni}{\cK^n_{\scriptscriptstyle 0}} 

\newcommand{\1}{\operatorname{\mathbbm{1}}}

\newcommand{\sln}{\operatorname{SL}(n)}

\newcommand{\relint}{\operatorname{relint}}

\newcommand{\lin}{\operatorname{lin}}
\newcommand{\aff}{\operatorname{aff}}

\begin{document}

\title{SL$(n)$ invariant valuations on polytopes}

\author{Monika Ludwig and Matthias Reitzner}

\begin{abstract}
A classification of $\sln$ invariant valuations on the space of convex polytopes in $\R^n$ without any continuity assumptions  is established. A corresponding result is obtained on the space of convex polytopes in $\R^n$ that contain the origin.
\end{abstract}

\maketitle 
\bigskip

\section{Introduction}

Since Felix Klein announced his Erlangen program nearly 150 years ago, the study and classification of invariants of geometric objects with respect to transformation groups are among  the most important tasks in geometry. In Euclidean space $\R^n$, volume and the Euler characteristic are invariant under translations and rotations. They are even invariant under maps from the special linear group, $\sln$. Both invariants turn out to have a further natural property. They satisfy the inclusion-exclusion principle showing that these functionals are valuations. Here, a functional  $\Psi$ is called a {\em valuation} on a collection $\mathcal S$ of sets if
\begin{equation}\label{valu}
\Psi(P)+\Psi(Q) = \Psi(P\cup Q) + \Psi(P\cap Q)
\end{equation}
whenever $P,Q, P\cap Q, P\cup Q \in \mathcal S$. Hence it is a natural task to classify invariants which are also valuations. 

The aim of this paper is to obtain a complete classification of $\sln$ invariants which are valuations on the set $\cP^n$ of convex polytopes in $\R^n$. We show that besides volume and the Euler characteristic there are further invariant valuations. We characterize these $\sln$ invariant valuations and prove that among these volume and Euler characteristic are essentially the only  invariants with certain continuity properties.

The classification of valuations using invariance and continuity properties is a classical part of geometry with important applications in integral geometry  (cf.\ \cite{Klain:Rota} and \cite[Chapter 6]{Schneider:CB2}). Such results are useful in the affine geometry of convex bodies and for affine invariant problems in analysis. 
As mentioned above, the  $n$-dimensional volume $V_n: \cP^n\to \R$ and the Euler characteristic $V_0:\cP^n\to \R$  (for which $V_0(P)=1$ for $P\ne\emptyset$) are the most important such functionals. Since we do not assume continuity, also functionals that depend on (possibly discontinuous) solutions $\psi:[0, \infty)\to \R$ of Cauchy's functional equation
$$\psi(x+y)=\psi(x)+\psi(y)$$
for $x,y\in[0,\infty)$ will occur.

Denote by $\cpni$ the subspace of convex polytopes that contain the origin.
First, we consider valuations defined on $\cpni$ and obtain the following result. Let $\dim P$ be the dimension of the polytope $P$, that is, the dimension of its affine hull, $\aff P$, and write $\relint P$ for the relative interior of $P$ with respect to $\aff P$. Let $n\ge 2$ throughout the paper.

\begin{theorem}\label{th:1}
A functional $\,\Psi:\cpni \to \R$ is an $\,\sln$ invariant valuation if and only if there are 
constants $c_0,c_0'\in\R$ 
and a solution  $\psi:[0,\infty)\to \R$ of Cauchy's functional equation such that 
$$\Psi (P)= c_0\,V_0(P) + c_0' \,(-1)^{\dim P\!}  \1_{\relint P} (0) +  \psi\big(V_n(P)\big)  $$
for every $P\in\cpni$.
\end{theorem}

\noindent
Here $\1_Q$ is the indicator function of $Q\subset \R^n$, that is, $\1_Q(x)=1$ if $x\in Q$ and $\1_Q(x)=0$ otherwise.

Let $\cP^n$ and $\cpni$ be equipped with the standard topology, which comes from the Hausdorff metric. 
A functional on $\cP^n$ or $\cpni$ is (Borel)\! {\em measurable} if the pre-image of every open set in $\R$ is a Borel set. In Section \ref{corollaries}, we show that $P\mapsto (-1)^{\dim P\!}  \1_{\relint P} (0)$ is measurable. It is well known that every measurable solution of Cauchy's functional equation is linear. This immediately implies the following result.

\begin{corol}\label{co:1a}
A functional $\,\Psi:\cpni \to \R$ is a measurable and $\,\sln$ invariant valuation if and only if there are constants $c_0, c_0', c_n \in\R$ such that 
\begin{equation}\label{eq:co:1a}
\Psi (P)= c_0\,V_0(P) +c_0'\, (-1)^{\dim P\!}  \1_{\relint P} (0) + c_n\,V_n(P) 
\end{equation}
for every $P\in\cpni$.
\end{corol}

Let $\cK^n$ denote the space of convex bodies (that is, compact convex sets) in $\R^n$ and let $\ckni$ be the subspace of convex bodies that contain the origin. On these spaces, important upper semi\-continuous functionals exist that are defined as certain curvature integrals. Classification results for $\sln$ invariant and upper semi\-continuous valuations were established in \cite{Ludwig:Reitzner, Ludwig:Reitzner2, Ludwig:affinelength}. 
For $\sln$ invariant valuations on  polytopes containing the origin,  the following result is a consequence of Corollary~\ref{co:1a}.

\begin{corol}\label{co:1b}
A functional $\,\Psi:\cpni \to \R$ is an upper semicontinuous and $\sln$ invariant valuation if and only if there are constants $c_0, c_n\in\R$ such that 
$$\Psi (P)= c_0\,V_0(P) +c_n\,V_n(P) $$
for every $P\in\cpni$.
\end{corol}

We remark that for upper semicontinuous and $\sln$ invariant valuations on $\cpnii$, that is, the set of convex polytopes that  contain the origin in their interiors,  a complete classification was established by Haberl \& Parapatits \cite{Haberl:Parapatits_centro}.  For homogeneous, measurable and $\sln$ invariant valuations such a result was established in \cite{Ludwig:origin}. Recently, Haberl \& Parapatits \cite{HaberlParapatits_moments} strengthened these results and obtained a complete classification of measurable and $\sln$ invariant valuations on $\cpnii$. 

\goodbreak

Next, we consider the space of all convex polytopes $\cP^n$. For $P\in\cP^n$, we write $[0,P]$ for the convex hull of the origin  and $P$.

\begin{theorem}\label{th:2}
A functional $\,\Psi:\cP^n \to \R$ is an $\,\sln$ invariant valuation if and only if there are constants $c_0, c_0', d_0\in\R$ and  solutions $\phi,\psi:[0,\infty)\to \R$ of Cauchy's functional equation such that 
$$\Psi (P)=  c_0 \,V_0(P) + c_0'\, (-1)^{\dim P\!}  \1_{\relint P} (0) + \psi\big(V_n(P)\big)+d_0 \1_{P}(0)
+ \phi\big( V_n([0,P])\big) $$
for every $P\in\cP^n$.
\end{theorem}

Taking into account that all measurable solutions of Cauchy's functional equation are linear immediately gives the following corollary.

\begin{corol}\label{co:1}
A functional $\,\Psi:\cP^n \to \R$ is a measurable and $\,\sln$ invariant valuation if and only if there are constants $c_0, c_0', c_n, d_0, d_n\in\R$  such that 
$$\Psi (P)=  c_0\, V_0(P) +c_0'\, (-1)^{\dim P\!}  \1_{\relint P} (0) + c_n\, V_n(P)+ d_0  \1_{P}(0) + d_n\, V_n([0,P]) $$
for every $P\in\cP^n$.
\end{corol}

\goodbreak
As in Corollary \ref{co:1b}, we also impose stronger assumptions on the valuations and obtain the following results.

\begin{corol}\label{co:2}
A functional $\,\Psi:\cP^n \to \R$ is an upper semicontinuous and $\,\sln$ invariant valuation if and only if there are constants $c_0,  c_n, d_n\in\R$ and $d_0\ge0$  such that 
$$\Psi (P)= c_0\,V_0(P) +  c_n\,V_n(P)+d_0 \1_{P}(0)+ d_n\,V_n([0,P])$$
for every $P\in\cP^n$.
\end{corol}

\begin{corol}\label{co:3}
A functional $\,\Psi:\cP^n \to \R$ is a continuous and $\,\sln$ invariant valuation if and only if there are constants $c_0,  c_n, d_n\in\R$   such that 
$$\Psi (P)= c_0\,V_0(P) +  c_n\,V_n(P)+ d_n\,V_n([0,P])$$
for every $P\in\cP^n$.
\end{corol}

We remark that deducing Theorem \ref{th:2} from Theorem \ref{th:1} is similar to the corresponding step for convex-body valued valuations. Classification results for convex-body valued valuations intertwining $\sln$ were first established on $\cK_0^n$ in \cite{Ludwig:Minkowski} (also see \cite{Haberl_sln}) and then extended to classification results on $\cK^n$ by Schuster \& Wannerer \cite{Schuster:Wannerer} and Wannerer~\cite{Wannerer2011}.

\section{Notation and Preliminaries}\label{tools}

We work in $n$-dimensional Euclidean space, $\R^n$, and denote its standard basis by $e_1,\ldots,e_n$. 
We write $\lin$ for linear hull and $[v_1, \dots, v_i]$ for the convex hull of $ v_1, \dots, v_i\in\R^n$.

\goodbreak
A valuation $\Psi:\cP^n\to\R$ can be extended to a valuation on finite unions of convex polytopes 
such that the inclusion-exclusion principle holds  
(cf.\ \cite{Klain:Rota} or \cite[Theorems 6.2.1 and 6.2.3]{Schneider:CB2}), that is,
for $P_1,\dots, P_m\in\cP^n$,
\begin{equation}\label{inex}
\Psi(P_1\cup\dots\cup P_m) = \sum_{j=1}^{m} (-1)^{j-1} \sum_{1 \leq i_1 < \dots < i_j \leq m} \Psi(P_{i_1} \cap \dots \cap P_{i_j}).
\end{equation}
This is also called {\em finite additivity}.

Let $Q$ be a finite union of $k$-dimensional polytopes. 
We define a triangulation of $Q$ into simplices as a set of $k$-dimensional simplices $\{T_1, \dots, 
T_m\}$ which have pairwise disjoint interiors, with $Q=\bigcup T_i $ and with the property that for 
arbitrary $1 \leq i_1 < \dots < i_j \leq m$ the 
intersections $T_{i_1} \cap \dots \cap T_{i_j}$ are again simplices. This guarantees that when making use of the inclusion-exclusion principle in this setting, on the right hand side only simplices occur.

We also require the following special case of a result by Haberl \cite[Lemma~3.2]{Haberl:star} (or see \cite[Lemma 2]{Haberl_sln}) for simple valuations, that is, for valuations which vanish on lower dimensional sets. 

\begin{lemma}[Haberl \cite{Haberl:star}]\label{christoph}
If $\,\Psi: \cpni \to \R$ is a simple valuation and $\Psi(T)=0$ for every $n$-dimensional simplex $T$ with one vertex at the origin, then $\Psi(P)=0$ for every $P\in\cpni$.
\end{lemma}
\goodbreak

\section{Proof of Theorem~\ref{th:1}}\label{pr:1}

First, we check that $\Psi: \cpni\to \R$ defined by $\Psi(P)=(-1)^{\dim P\!}  \1_{\relint P} (0)$ is 
a valuation, that is, we check that (\ref{valu}) holds for all $P,Q\in\cpni$ with $P\cup Q \in 
\cpni$. Note that (\ref{valu}) holds if $P\subseteq Q$ or $Q\subseteq P$. If there is no inclusion, 
then $P\cup Q \in \cpni$ implies that $P$ and $Q$ have the same affine hull and hence $\dim P=\dim 
Q=\dim(P\cup Q)$. If $0\in \relint P$ and $0\in \relint Q$, then $0\in \relint(P\cap Q)$ and $0\in 
\relint(P\cup Q)$ and  therefore (\ref{valu}) holds. If $0\in \relint P$ and $0\not\in\relint Q$ (or 
vice versa), then $0\in \relint(P\cup Q)$ and $0\not\in \relint (P\cap Q)$ and therefore 
(\ref{valu}) holds. If $0\not\in\relint P$ and $0\not\in\relint Q$ while $0\not\in\relint (P\cup 
Q)$, then $0\not\in\relint(P\cap Q)$ and therefore (\ref{valu}) holds. Finally, if $0\not\in\relint 
P$ and $0\not\in\relint Q$ while $0\in\relint (P\cup Q)$, then $0\in\relint(P\cap Q)$ and 
$\dim(P\cup Q)= \dim(P\cap Q) +1$ and therefore (\ref{valu}) holds. Thus $\Psi$ is a valuation and 
it clearly is $\sln$ invariant. Hence 
for $c_0,c_0'\in\R$ and $\psi:[0,\infty)\to \R$ a solution of Cauchy's functional equation 
$$P \mapsto  c_0\,V_0(P) + c_0' \,(-1)^{\dim P\!}  \1_{\relint P} (0) +  \psi\big(V_n(P)\big)  $$
is an $\sln$ invariant valuation on $\cpni$.
We have to show that every $\sln$ invariant valuation $\Psi: \cpni\to\R$ is of such form.

For $i=0, \dots , n$, let $\cT^i$ be the set of $i$-dimensional simplices $T \subset \R^n$ with one vertex at the origin $0$. 

\begin{lemma}\label{le:simplex}
If\, $\Psi: \cpni\to \R$ is an $\sln$ invariant valuation, then there are a constant $c_0\in\R$ and a solution $\psi: [0,\infty) \to \R$ of Cauchy's functional equation such that
$$ \Psi(T) = c_0\, V_0(T)  + \psi\big(V_n(T)\big) $$
for every $\,T\in \cT^1 \cup\dots\cup \cT^n$.
\end{lemma}

\begin{proof}
Note that for $k \leq n-1$, any simplex $T \in \cT^k$ is an $\sln$ image of the simplex $[0, e_1, \dots, e_k]$.
Set 
$$ c_0 = \Psi([0, e_1]) . $$
We show that
\begin{equation}\label{ind}
\Psi(T)=c_0\,\,\, \mbox{ for all } \,\,T \in \cT^1 \cup \dots \cup \cT^{n-1}.
\end{equation}
By definition this holds true for $T \in \cT^1$. 

For $n\geq 3$, assume $T_1, T_2, T_1 \cup T_2 \in \cT^{2}$ and $T_1\cap T_2\in \cT^1$. By the additivity of $\Psi$, we have
$$\Psi(T_1) + \Psi(T_2) = \Psi(T_1 \cup T_2) + \Psi(T_1 \cap T_2) .$$
Since all simplices in $\cT^{2}$ are $\sln$ images of $[0, e_1, e_2]$, we obtain
$$\Psi([0,e_1, e_{2}])= \Psi(T_1 \cap T_2) = \Psi([0, e_1]) = c_0.$$
We continue by induction. Assume that for $k \leq n-2$ we already know that 
$\Psi(T)=c_0$ for all $T \in \cT^1 \cup \dots \cup \cT^{k-1}$. 
Further,  let  $T_1, T_2, T_1 \cup T_2 \in \cT^{k}$ and $T_1\cap T_2\in\cT^{k-1}$. By the additivity of $\Psi$, we have
$$\Psi(T_1) + \Psi(T_2) = \Psi(T_1 \cup T_2) + \Psi(T_1 \cap T_2) .$$
Since all simplices in $\cT^{k}$ are $\sln$ images of $[0, e_1, \dots, e_{k}]$ for $k \leq n-1$, we obtain
$$\Psi([0,e_1, \dots, e_{k}])= \Psi(T_1 \cap T_2) = \Psi([0, e_1, \dots, e_{k-1}]) = c_0.$$
This proves (\ref{ind}).

\goodbreak
In the last step, let $T \in \cT^{n}$.
There is  $A \in\sln$ such that 
$$A T= [0, e_1, \dots, e_{n-1},n! v\, e_n]  $$
where $v= V_n(T)$. Define $\alpha:[0,\infty)\to \R$ by
$$\alpha(v)= \Psi([0, e_1, \dots, e_{n-1}, n! v\, e_n]) $$
and note that  by the $\sln$ invariance $\Psi(T)=  \alpha\big(V_n(T)\big)$ for all  $T \in \cT^n$.
\goodbreak

For $x, y\ge 0$, let $T\in \cT^n$ be such that $x+y =V_n(T)$. We choose $ T_1, T_2 \in\cT^n$  such 
that  $x= V_n(T_1)$ and $y= V_n(T_2)$ 
while $T=T_1\cup T_2$ and $T_1, T_2$ have disjoint interiors. It follows that
$$\alpha(x+y)=\Psi(T) = \Psi(T_1) + \Psi(T_2) - \Psi(T_1 \cap T_2) = \alpha(x)+\alpha(y) - c_0.$$
Thus $\alpha- c_0$ is a solution, say $\psi:[0,\infty)\to \R$, of Cauchy's functional equation. 
Thus we obtain
$$ \Psi(T) =c_0\, V_0(T) + \psi \big(V_n(T)\big) $$
which proves the lemma. 
\end{proof}
\goodbreak

Let $c_0$ and  $\psi$ be as in Lemma~\ref{le:simplex} and set $c_0'= \Psi(\{0\})-c_0$.
Define  $\Psi': \cpni\to \R$ as
$$ \Psi' = \Psi-  c_0\, V_0 - \psi\circ V_n - c_0' (-1)^{\dim P} \1_{\relint P}(0). $$
Note that $\Psi'$ is an  $\sln$ invariant valuation, that $\Psi'(\{0\})=0$ and that $\Psi'$ vanishes by Lemma~\ref{le:simplex} on $\cT^1\cup\dots\cup\cT^{n}$. The following lemma completes the proof of the theorem.

\begin{lemma}
The valuation $\Psi'$ vanishes on $\cpni$.
\end{lemma}

\begin{proof}
Since $\Psi'$ vanishes on $\cT^0\cup \cT^1$,  for $s,t>0$ we have
$ \Psi'([0,t e_1])= 0 $
and 
$$ \Psi'([-s e_1,t e_1])= \Psi'([-s e_1,0]) + \Psi'([0,t e_1])  = 0 . $$
Thus $\Psi'$ vanishes on all at most 1-dimensional polytopes  in $\cpni$. 

We proceed by induction on $k=\dim P$ and assume that $\Psi'(P)=0$ holds for all at most $(k-1)$-dimensional $P\in\cpni$. Hence $\Psi'$ is simple when restricted to  polytopes in $\cpni$ in a $k$-dimensional subspace. Since $\Psi'$ vanishes on $\cT^k$, Lemma~\ref{christoph} implies that $\Psi'$ vanishes on all polytopes in $\cpni$ in this $k$-dimensional subspace.  Hence $\Psi'$ vanishes on all at most $k$-dimensional polytopes in $\cpni$. This completes the proof of the lemma.
\end{proof}

\section{Proof of Theorem \ref{th:2}}
It is easy to check (as in the proof of Theorem \ref{th:1}) that $P\mapsto (-1)^{\dim P\!}  \1_{\relint P} (0)$ is a valuation on $\cP^n$  and also $P\mapsto \1_{P}(0)$ is a valuation on $\cP^n$. They clearly are $\sln$ invariant. Hence 
for $c_0, c_0', d_0\in\R$ and  $\phi,\psi:[0,\infty)\to \R$ solutions of Cauchy's functional equation  
$$P\mapsto c_0 \,V_0(P) + c_0'\, (-1)^{\dim P\!}  \1_{\relint P} (0) + \psi\big(V_n(P)\big)+d_0 \1_{P}(0)
+ \phi\big( V_n([0,P])\big) $$
is an $\sln$ invariant valuation on $\cP^n$.
We have to show that every $\sln$ invariant valuation $\Psi: \cP^n\to\R$ is of such form.

We apply Theorem \ref{th:1} to the restriction of $\Psi$ to $\cpni$ and obtain $a_0, a_0'\in\R$ and a solution $\alpha:[0,\infty)\to\R$ of Cauchy's functional equation such that
\begin{equation*}\label{zwei}
\Psi(P)= a_0\,V_0(P) + a_0' \,(-1)^{\dim P\!}  \1_{\relint P} (0) +  \alpha\big(V_n(P)\big)
\end{equation*}
for $P\in\cpni$.
Set $b_0 = \Psi(\{e_1\})$.
Define $\Psi':\cP^n\to\R$ by
$$\Psi' (P)= \Psi(P) 
- a_0\,\1_P (0) - a_0'\, (-1)^{\dim P\!}  \1_{\relint P} (0) - b_0 \1_{P^c}(0)
- \alpha\big(V_n(P)\big), $$
where $P^c$ is the complement of $P$ in $\R^n$.
Note that $\Psi'$ is an $\sln$ invariant valuation on $\cP^n$ which vanishes on $\cpni$.

First we consider $\Psi'$ on at most $(n-2)$-dimensional polytopes.

\begin{lemma}\label{l:n-2}
The valuation $\Psi'$ vanishes on every polytope $P \in \cP^n$ with $\dim P \leq n-2$.
\end{lemma}

\begin{proof}
Note that $\Psi'$ vanishes already on $\cpni$ and thus we have to take care of polytopes $P$ in $\cP^n \setminus \cpni$.
We prove the statement by induction on $k=\dim P$. 
For $k=0$, we have $\Psi'(\{x\})=\Psi'(\{e_1\})=0$ for $x \neq 0$.
Assume $ \Psi'(P)= 0$ for all $P\in\cP^n$ with $\dim P\leq k-1$. We prove the statement for $\dim P =k\le n-2$.

First, let  $T$ be a $k$-dimensional simplex and $0 \notin \aff T$. There is a special linear map from $T$ onto $[e_1, e_2, \dots,  e_{k+1}]$ and if we dissect $T$ into two $k$-dimensional simplices $T_1$ and  $T_2$, then there are special linear maps from $T_1$ and $T_2$ onto $[e_1, e_2, \dots,  e_{k+1}]$. By the $\sln$ invariance of $\Psi'$, the valuation property and the induction assumption,  we obtain
$$\Psi'(T)= \Psi'(T_1)+\Psi'(T_2) - \Psi'(T_1 \cap T_2)= 2 \Psi'(T)  $$
which proves $ \Psi'(T)=0$.
\goodbreak

Second, let $P$ be a $k$-dimensional polytope with $0 \notin \aff P$. Triangulate  $P$ into $k$-dimensional simplices $T_1, \dots, T_m$. By the inclusion-exclusion principle (\ref{inex}), the induction assumption and the statement just proved for simplices, we have
$$
\Psi'(P)
=
\sum_{j=1}^{m} (-1)^{j-1} \sum_{1 \leq i_1 < \dots < i_j \leq m} \Psi'(T_{i_1} \cap \dots \cap T_{i_j}) 
=0.
$$

Third, let  $P$ be  a $k$-dimensional polytope with $0 \in \aff P $. For $P \in \cpni$ we already have $\Psi'(P)=0$. So assume $0 \notin P$ and let $F_1, \dots, F_m$ be the facets of $P$ visible from the origin, i.e. $ \relint F_i \subseteq \relint [0,P]$. 
Triangulate  the facets $F_i$ into simplices $T'_1, \dots, T'_l$,  and thus the closure of $[0,P]\setminus P$  into simplices 
$T_1=[0, T'_1], \dots, T_l=[0, T'_l]$ with a vertex at the origin.  Using the inclusion-exclusion principle (\ref{inex}),  the fact  that $\Psi'$ vanishes on $\cpni$, and by the induction hypothesis also on polytopes of dimension at most $k-1$, we obtain 
\begin{eqnarray*}
0 = \Psi'(\underbrace{[0,P]}_{\in \cpni})\!\!\!
&=& 
\sum_{j=1}^{m} (-1)^{j-1} \sum_{1 \leq i_1 < \dots < i_j \leq m} \Psi'(\underbrace{T_{i_1} \cap \dots \cap T_{i_j}}_{\in \cpni}) 
\\ && +
\sum_{j=1}^{m} (-1)^{j} \sum_{1 \leq i_1 < \dots < i_j \leq m} \Psi'(\underbrace{T_{i_1} \cap \dots \cap T_{i_j} \cap P}_{\dim \leq k-1}) + \Psi'(P)
\\ &=&
\Psi'(P).
\end{eqnarray*}
This completes the proof of the lemma.
\end{proof}

It remains to investigate polytopes of dimension at least $n-1$.

\goodbreak
\begin{lemma}\label{l:n-1}
There is a solution  $\beta:[0,\infty)\to\R$ of Cauchy's functional equation such that 
$$ \Psi'(P)=\beta\big( V_n([0,P])\big) $$
for every $(n-1)$-dimensional polytope $P \in\cP^n$.
\end{lemma}

\begin{proof}
First, let $T$ be an $(n-1)$-dimensional simplex with  $0 \notin \aff T$.  There is a  special linear map from $T$ onto the simplex $[e_1, \dots, e_{n-1}, n! v\, e_{n}]$ with $v= V_n([0, T])$. Define the function $\beta: [0,\infty)\to \R$ by
$$\beta(v) =\Psi'([e_1, \dots, e_{n-1}, n! v\, e_{n}]) = \Psi'(T).$$
Dissecting $T$ into two $(n-1)$-dimensional simplices $T_1$ and $T_2$ and setting $v_i=V_n([0, T_i])$ for $i=1,2$, we obtain by Lemma \ref{l:n-2} that
$$ \beta(v)= \Psi'(T)= \Psi'(T_1)+\Psi'(T_2) = \beta(v_1) + \beta(v_2)$$
where clearly $v=v_1+v_2$.
Thus $\beta$ is a solution of Cauchy's functional equation, and  we have $\Psi'(T)= \beta\big(V_n([0,T])\big)$.

\goodbreak
Second, let $P$ be an $(n-1)$-dimensional polytope with $0 \notin \aff P$. Triangulate $P$  into simplices $T_1, \dots, T_l$. Using the inclusion-exclusion principle (\ref{inex}) and that $\Psi'$ vanishes on at most $(n-2)$-dimensional polytopes, we obtain 
$$\Psi'(P)=\sum_{j=1}^{l} \Psi'(T_j)=\sum_{j=1}^{l} \beta\big(V_n([0,T_j]\big)=
\beta\big( V_n([0,P])\big), 
$$
where we used the previous calculation for simplices  and the finite additivity of $V_n$.

Third, let $P$ be an $(n-1)$-dimensional polytope with $0 \in \aff P$. Then the polytope $[0,P]$ is $(n-1)$-dimensional and thus $\beta \big(V_n([0,P])\big)=0$. So we have to prove that $\Psi'(P)=0$. If $P \in \cpni$ we already know that $\Psi'(P)=0$. So assume $0 \notin P$, and as in the proof of Lemma~\ref{l:n-2} triangulate the facets of $P$ visible from the origin,  and thus the closure of $[0,P]\setminus P$  into simplices 
$T_1, \dots, T_l$ with a vertex at the origin. Using that $\Psi'$ vanishes on $\cpni$ and on at most $(n-2)$-dimensional polytopes and the inclusion-exclusion principle (\ref{inex}), we obtain  
$$0 = \Psi'([0,P])=\sum_{j=1}^{l}  \Psi'(T_j )  +\Psi'(P)= \Psi'(P)$$ 
which completes the proof of the lemma.
\end{proof}
Observe that $V_n([0,P])=V_n([0,P]\setminus P)$ for every $(n-1)$-dimensional polytope $P \in\cP^n$. Thus Lemmas~\ref{l:n-2} and \ref{l:n-1} yield
$$ \Psi'(P) = \beta\big( V_n([0,P]\setminus P)\big) $$
if $\dim P \leq n-1$.
It remains to prove this also for polytopes $P$ of dimension $n$. 

For $P \in \cpni$, the assertion is trivial since $\Psi'$ vanishes on $\cpni$.  
Let $P\in\cP^n\backslash \cpni$, and let $F_1, \dots, F_m$ be the facets of $P$ visible from the 
origin. Since $\Psi'$ vanishes on $\cpni$ and on all at most $(n-2)$-dimensional polytopes, we have 
by the inclusion-exclusion principle (\ref{inex})
\begin{eqnarray*}
0 &=& \Psi'([0,P])
\\ &=&
\sum_{j=1}^{m} (-1)^{j-1} \sum_{1 \leq i_1 < \dots < i_j \leq m} \Psi'(\underbrace{[0, F_{i_1}] 
\cap \dots \cap [0,F_{i_j}]}_{\in \cpni}) 
\\ && +\ 
\sum_{j=2}^{m} (-1)^{j} \sum_{1 \leq i_1 < \dots < i_j \leq m} \Psi'(\underbrace{[0,F_{i_1}] \cap 
\dots \cap [0,F_{i_j}] \cap P}_{\dim \leq n-2}) 
\\ && -  \sum_{1 \leq i \leq m} \Psi'(\underbrace{[0,F_{i}]  \cap P}_{= F_i}) + 
\Psi'(P)
\\ &=&
\Psi'(P)- \sum_{i=1}^m \Psi'( F_i).
\end{eqnarray*}
Hence Lemma \ref{l:n-1} implies that
$$\Psi'(P)=\sum_{i=1}^m\Psi'(F_i) =\sum_{i=1}^m\beta\big(V_n([0,F_i])\big) = \beta\big(V_n([0,P]\backslash P)\big).$$
Thus we obtain the following for $\Psi$. For all $P\in\cP^n$,
\begin{align*}
\Psi(P)= \,&\, 
a_0\,\1_P (0) + a_0'\, (-1)^{\dim P\!}  \1_{\relint P} (0) + b_0 \1_{P^c}(0) 
\\[3pt]  &
+ \alpha\big(V_n(P)\big) + \beta\big(V_n([0,P]\backslash P)\big) .
\end{align*}
Set $c_0= a_0+b_0$ and $c_0'=a_0'$ as well as $d_0=-b_0$. Further define the functions $\phi, \psi:[0,\infty)\to \R$ as $\psi=\alpha-\beta$ and $\phi=\beta$. This gives the representation from Theorem \ref{th:2}. 

z
\section{Proofs of the Corollaries}\label{corollaries}
For Corollaries~\ref{co:1a} and \ref{co:1}, we show that the function $P\mapsto (-1)^{\dim P\!}  
\1_{\relint P} (0)$ is measurable. For $k=0,\dots, n$, the set  $\{P\in \cpni: \dim P\le k\}$ is 
closed in $\cpni$. Hence $\cP_{\scriptscriptstyle 0,k}^n=\{P\in \cpni: \dim P= k\}$ is a Borel set 
in $\cpni$ for $k=0,\dots, n$. Since the set $\{P \in \cpni:\dim P=k, \ 0 \in \relint P\} $ is open 
in $\cP_{\scriptscriptstyle 0,k}^n$, this shows that 
$\{P \in \cpni:\dim P=k, \ 0 \in \relint P\} $ is a Borel set in $\cpni$ and in $\cP^n$ for 
$k=0,\dots, n$, which implies measurability.

For Corollaries~\ref{co:1b} and \ref{co:2}, it suffices to show that the valuation
$$\Phi (P)= c_0' (-1)^{\dim P}  \1_{\relint P} (0)$$
is upper semicontinuous if and only if $c_0'=0$. 
Indeed, 
$$ \lim_{s \to 0}  \Phi([- se_1, se_1]) = - c_0' \leq \Phi(0) =c_0' $$
and hence $c_0' \geq 0$.
On the other hand
$$ \lim_{s \to 0}  \Phi([- se_1, se_1, -e_2, e_2]) = c_0' \leq \Phi([-e_2, e_2]) =-c_0' $$
and thus $c_0' \leq 0$ which gives $c_0'=0$.

For Corollary~\ref{co:2} it remains to note that $P\mapsto d_0 \1_{P}(0)$ is upper semicontinuous for $d_0 \geq 0$. 

\subsection*{Acknowledgments}
The authors thank the anonymous referees for many helpful comments. The work of Monika Ludwig was supported, in part, by Austrian Science Fund (FWF) Project P25515-N25.

\medskip
\parindent=0pt

\bigskip
\begin{samepage}
Monika Ludwig\\
Institut f\"ur Diskrete Mathematik und Geometrie\\
Technische Universit\"at Wien\\
Wiedner Hauptstra\ss e 8-10/1046\\
1040 Wien, Austria\\
e-mail: monika.ludwig@tuwien.ac.at
\end{samepage}
\bigskip

\begin{samepage}
Matthias Reitzner\\
Institut f\"ur Mathematik\\
Universit\"at  Osnabr\"uck\\
Albrechtstra\ss e 28a \\
49076 Osnabr\"uck, Germany\\
e-mail: matthias.reitzner@uni-osnabrueck.de
\end{samepage}

\bigskip

\end{document}